\documentclass{article}
\usepackage[T1]{fontenc}
\usepackage{geometry}
\usepackage{amssymb}
\usepackage{amsmath,cases}
\usepackage{amsthm}
\usepackage{mathrsfs}
\usepackage{enumitem}
\usepackage{xcolor}
\usepackage{graphicx}
\usepackage{textcomp}
\usepackage{caption}
\usepackage{url}
\usepackage{calc}
\usepackage[active]{srcltx}   % Inverse Search
\usepackage{pdfsync}
\usepackage{todonotes}
\usepackage[colorlinks=true,bookmarksnumbered=true, pdfauthor={J\'er\'emie Bettinelli, Eric Fusy, C\'ecile Mailler, Lucas Randazzo}]{hyperref}									%à mettre en dernier

\newcommand{\lv}[1]{{#1}}
\newcommand{\sv}[1]{}

\renewcommand{\emptyset}{\varnothing}
\newcommand{\eps}{\varepsilon}
\newcommand{\si}{\sigma}
\newcommand{\de}{\mathrel{\mathop:}\hspace*{-.6pt}=}
\newcommand{\boxd}[1]{\framebox{\vphantom1\ensuremath{#1}}}
\newcommand{\un}{\mathbf{1}}
\newcommand{\sq}{\operatorname{Seq}}
\newcommand{\squ}{\sq_{\ge 1}}

\newcommand{\sqk}{\sq_{\ge k}}

\newcommand{\cA}{\mathcal{A}}
\newcommand{\cB}{\mathcal{B}}
\newcommand{\cC}{\mathcal{C}}
\newcommand{\cG}{\mathcal{G}}
\newcommand{\cM}{\mathcal{M}}
\newcommand{\cP}{\mathcal{P}}
\newcommand{\cR}{\mathcal{R}}
\newcommand{\cT}{\mathcal{T}}
\newcommand{\cX}{\mathcal{X}}
\newcommand{\cZ}{\mathcal{Z}}

\newcommand{\sN}{\mathscr{N}}
\newcommand{\sS}{\mathscr{S}}

\newcommand{\fa}{\mathfrak{a}}
\newcommand{\fb}{\mathfrak{b}}
\newcommand{\fc}{\mathfrak{c}}
\newcommand{\fp}{\mathfrak{w}}

\newcommand{\C}{\operatorname{Cat}}

\theoremstyle{plain}
\newtheorem{thm}{Theorem}
\newtheorem{lem}[thm]{Lemma}
\newtheorem{prop}[thm]{Proposition}

\newtheorem{defi}{Definition}

\theoremstyle{definition}
\newtheorem*{rem}{Remark}
\newtheorem*{note}{Note}
\newtheorem*{notn}{Notation}
\newtheorem*{ack}{Acknowledgment}

\newenvironment{pre}[1][\proofname]{%
  \proof[#1]%
}{\endproof}

\makeatletter
\renewcommand*{\@fnsymbol}[1]{\ensuremath{\ifcase#1\or \dagger\or \ddagger\or
   \mathsection\or \mathparagraph\or \|\or **\or \dagger\dagger
   \or \ddagger\ddagger \else\@ctrerr\fi}}
\makeatother
\title{A bijective study of basketball walks}
\author{J\'er\'emie \textsc{Bettinelli}\thanks{CNRS \& Laboratoire d'Informatique de l'\'Ecole polytechnique; \href{mailto:jeremie.bettinelli@normalesup.org}{\nolinkurl{jeremie.bettinelli@normalesup.org}};\newline \nolinkurl{www.normalesup.org/}\texttildelow\nolinkurl{bettinel}.}
\and 
\'Eric \textsc{Fusy}\thanks{CNRS \& Laboratoire d'Informatique de l'\'Ecole polytechnique; \href{mailto:fusy@lix.polytechnique.fr}{\nolinkurl{fusy@lix.polytechnique.fr}};\newline \nolinkurl{www.lix.polytechnique.fr/Labo/Eric.Fusy/}.}
\and
C\'ecile \textsc{Mailler}\thanks{Department of Mathematical Sciences, University of Bath; \href{mailto:c.mailler@bath.ac.uk}{\nolinkurl{c.mailler@bath.ac.uk}}; \nolinkurl{people.bath.ac.uk/cdm37/}.}
\and
Lucas \textsc{Randazzo}\thanks{Laboratoire d'informatique Gaspard-Monge, Universit\'e Paris-Est Marne-la-Vall\'ee; \href{mailto:lucas.randazzo@u-pem.fr}{\nolinkurl{lucas.randazzo@u-pem.fr}}.}
}

\begin{document}
\maketitle

\begin{abstract}
The Catalan numbers count many classes of combinatorial objects. The most emblematic such objects are probably the Dyck walks and the binary trees, and, whenever another class of combinatorial objects is counted by the Catalan numbers, it is natural to search for an explicit bijection between the latter objects and one of the former objects. In most cases, such a bijection happens to be relatively simple but it might sometimes be more intricate.

In this work, we focus on so-called \emph{basketball walks}, which are integer-valued walks with step-set $\{-2,-1,+1,+2\}$. 
%The presence of~$-2$ as an allowed step makes it impossible to use the classical {\L}ukasiewicz encoding of trees by integer-valued walks, and thus a different strategy is needed. 
We give an explicit bijection that maps, for each $n\ge 2$, $n$-step basketball walks from~$0$ to~$0$ that visit~$1$ and are positive except at their extremities to $n$-leaf binary trees. Moreover, we can partition the steps of a walk into $\pm 1$-steps, odd $+2$-steps or even $-2$-steps, and odd $-2$-steps or even $+2$-steps, and these three types of steps are mapped through our bijection to double leaves, left leaves, and right leaves of the corresponding tree.

We also prove that basketball walks from~$0$ to~$1$ that are positive except at the origin are in bijection with increasing unary-binary trees with associated permutation avoiding~$213$. We furthermore give the refined generating function of these objects with an extra variable accounting for the unary nodes.
\end{abstract}

%\tableofcontents

\section{Introduction}\label{secintro}

An integer-valued \emph{walk} is a finite sequence of integers, usually starting at~$0$. Its \emph{steps} are the differences of two consecutive values it takes. Walks with a given possible step-set have been the focus of many studies \cite{gessel1980factorization,labelle1990generalized,duchon2000enumeration,banderier2002basic, Bousquet08basket,Kra15} and so-called {\L}ukasiewicz walks (walks whose step-set~$\sS$ satisfies $-1\in\sS\subseteq\{-1,0,1,2,\ldots\}$) are of particular interest as they encode plane trees whose vertex outdegrees (or arities) all lie in $\sS+1$ \cite[Chapter~11]{lothaire1997combinatorics}.

In this work, we consider the simplest case of walks that are not {\L}ukasiewicz walks, namely \emph{basketball walks}. These are walks with step-set $\{-2,-1,+1,+2\}$; they were named by Ayyer and Zeilberger~\cite{AyZe07} from the fact that they record the score difference between two teams playing a basketball game at a time where three-pointers did not exist. Basketball excursions, that is, nonnegative basketball walks starting and ending at~$0$, were first counted by Labelle and Yeh~\cite{labelle1989dyck,labelle1990generalized} (they are called Dyck paths of knight moves in this reference). The authors give an explicit algebraic (quartic) expression for their generating function. Motivated by their connections with constrained polymers problems, Ayyer and Zeilberger~\cite{AyZe07} studied the generating function of bounded basketball excursions. Basketball walks were also later considered by Bousquet-M\'elou~\cite{Bousquet08basket}, who looked at the structure of the algebraic equations satisfied by the generating function of excursions, as well as that of bounded excursions. The latter has also been studied recently~\cite{bacher2013generalized}.

More recently, using the so-called kernel method, \lv{Banderier, Krattenthaler, Krinik, Kruchinin, Kruchinin, Nguyen and Wallner}\sv{the authors of}~\cite{BKKKKNW} \lv{showed that}\sv{studied} the generating function \sv{$G$ }of basketball walks from~$0$ to~$1$ that are positive except at the origin, counted with weight~$z$ per step\lv{, is given\footnote{In~\cite{BKKKKNW}, the authors use the notation~$G_{0,1}$. We removed the subscript in order to lighten the notation.} by
\begin{equation*}\label{G01}
G(z)=-\frac12+\frac12\sqrt{\frac{2 - 3z - 2\sqrt{1 - 4z}}{z}}.
\end{equation*}
}\sv{. }They also derived summatory expressions for the number of such walks of given length, and for the number of basketball excursions of a given length, as well as asymptotics for these numbers. Finally, they noted, and this was the starting point of the present work, that the generating function~$G$ is related to the Catalan generating function~$\C$, characterized by $\C(z) = 1 + z\C(z)^2$, by the simple equation~\cite[Equation~(3.14)]{BKKKKNW}:
\begin{equation}\label{eqBKKKKNW}
1 + G(z) + G^2(z) = \C(z).
\end{equation}

Instead of considering $G$-walks, we will rather focus on basketball walks from~$0$ to~$0$ that visit~$1$ and are positive except at the extremities: we call them $C$-walks and denote by~$C$ their generating function. A first step decomposition of $C$-walks yields that $C(z)=zG(z) + zG^2(z)$ as a $C$-walk is either a $+1$-step followed by a reversed $G$-walk or a $+2$-step followed by a reversed $G$-walk and another reversed $G$-walk (recall that $C$-walks have to visit~$1$). Expressed in terms of~$C$, Equation~\eqref{eqBKKKKNW} becomes
\begin{equation}\label{eqC}
C(z)=z\,\big(\!\C(z)-1\big),
\end{equation}
which is the generating function of nontrivial binary trees counted with weight~$z$ per leaf. This is the equation to which we give a bijective interpretation. As a byproduct of our bijection, we obtain the following proposition. We say that a step of a walk is \emph{even} (resp.\ \emph{odd}) if it starts at even (resp.\ odd) height.

\begin{prop}\label{propstats}
The number of basketball walks from~$0$ to~$0$ that visit~$1$ and are positive except at the extremities, with~$2d$ $\pm 1$-steps, $\ell$ odd $+2$-steps or even $-2$-steps, and~$r$ odd $-2$-steps or even $+2$-steps is equal to
$$\frac{1}{d}\binom{2d-2}{d-1}\binom{\ell+r+2d-2}{\ell+r}\binom{\ell+r}{\ell}\,.$$
\end{prop}

\bigskip
In the second half of this article, we show how basketball walks are related to increasing unary-binary trees with associated permutation avoiding~$213$. Recall that a \emph{unary-binary tree} of size $n\ge 1$ is a plane tree with~$n$ vertices, each having~$0$, $1$ or $2$ children. It is \emph{increasing} if its vertices are bijectively labeled with~$1$, \ldots, $n$ in such a way that any vertex receives a label larger than its parent's. In other words, the sequence of labels read along any branch of the tree from the root to a leaf is increasing. With an increasing tree, we associate the permutation obtained by reading the labels of the tree in breadth-first search order, from left to right. Recall that a permutation $\si=(\si_1,\ldots,\si_n)$ is said to \emph{avoid} the pattern $\pi=(\pi_1,\ldots,\pi_k)$ if, for any $1\le i_1 < i_2 < \ldots < i_k\le n$, the permutation obtained from $(\si_{i_1},\ldots,\si_{i_k})$ after relabeling with $1$, \ldots, $k$ in the same relative order is different from~$\pi$.\lv{ (Note that, in particular, a permutation of size~$n$ avoids any pattern of size strictly larger than~$n$.)}

The study of increasing trees whose associated permutation avoids a given pattern was initiated by Riehl (see for instance~\cite{LPRS16heaps}). By computer programming, she observed that the first terms of the sequence of increasing unary-binary trees with associated permutation avoiding~$213$ coincide with those of the sequence referenced \href{https://oeis.org/A166135}{\nolinkurl{A166135}} in the On-Line Encyclopedia of Integer Sequences, which was later proved in~\cite{BKKKKNW} to be the sequence with generating function~$G$. We show in the present work that this is indeed the case and thus answer \cite[Conjecture~5.2]{BKKKKNW}.

\begin{thm}\label{thmubt}
For any $n\ge 0$, the number of $n$-step basketball walks from~$0$ to~$1$ that are positive except at the origin is equal to the number of $n$-vertex increasing unary-binary trees with associated permutation avoiding~$213$.
\end{thm}

\sv{In the extended version of this paper, we}\lv{We will} furthermore explain how to obtain an explicit bijection between the classes of objects appearing in Theorem~\ref{thmubt}. It is well known that permutations avoiding~$213$ are yet another class of combinatorial objects counted by Catalan numbers. We are thus interested in a family of trees carrying an extra combinatorial structure counted by Catalan. Of course, both structures are strongly dependent because of the condition of being increasing. \lv{We will however see in Lemma~\ref{lembits} that}\sv{In fact,} the two structures can be separated and the conditions the permutation has to fulfill are completely encoded in the sequence of bits obtained by reading the tree in breadth-first search order and recording whether the vertices encountered are nodes or leaves\footnote{Recall that a vertex is called a \emph{node} if it has at least one child and a \emph{leaf} otherwise.}.

As an immediate consequence of this theorem, for $n\ge 1$, the following expressions found in~\cite[Proposition~3.5]{BKKKKNW}
$$[z^n]G(z) =\frac{1}{n} \sum_{k=1}^{n} (-1)^{k+1} \binom{2k-2}{k-1} \binom{2n}{n-k}	= \frac{1}{n}\sum_{i=0}^{n}\binom{n}{i}\binom{n}{2n+1-3i}$$
also give the number of $n$-vertex increasing unary-binary trees with associated permutation avoiding~$213$. Moreover, \cite[Theorem~3.8]{BKKKKNW} gives an asymptotic estimate of this number.

Finally, we study a refinement of the previous formula, which takes as an extra parameter the number of unary nodes, that is, nodes with only one child. 

\begin{prop}\label{unaire}
For $n\ge 1$ and $0\le k\le \lfloor (n-1)/2\rfloor$, the number of $n$-vertex increasing unary-binary trees with associated permutation avoiding~$213$ that have exactly $n-1-2k$ unary nodes is equal to
$$\frac1n\,\binom{2n}{k} \binom{n-k}{k+1}\,.$$
In particular, we obtain yet another expression for $[z^n]G(z)$:
$$[z^n]G(z)=\frac{1}{n} \sum_{k\ge 0} \binom{2n}{k} \binom{n-k}{k+1}\,.$$
\end{prop}

In the light of Theorem~\ref{thmubt}, it is natural to try and find the statistics corresponding to unary nodes on basketball walks: it is the number of what we call \emph{staggered $\pm2$-steps} (see Figure~\ref{fig:staggered}).

\begin{defi}\label{def:staggered}
Let $(w_0,w_1, \ldots, w_n)$ be the successive heights taken by a basketball walk and let, for all integers $1\le i\le n$, $u_i\de w_i-w_{i-1}$ be the value of its $i$-th step. For $i<j$, we say that~$u_i$ and~$u_j$ are paired if $u_i=+2$, $u_j=-2$, $w_j = w_{i-1}$, and for all $i<k<j$, we have $w_k\ge w_{i-1}+1$. A \emph{staggered $\pm 2$-step} is a $\pm2$-step that is not paired with any other $\pm2$-step.
\end{defi}

\begin{figure}[ht]
	\centering\includegraphics[width=.4\linewidth]{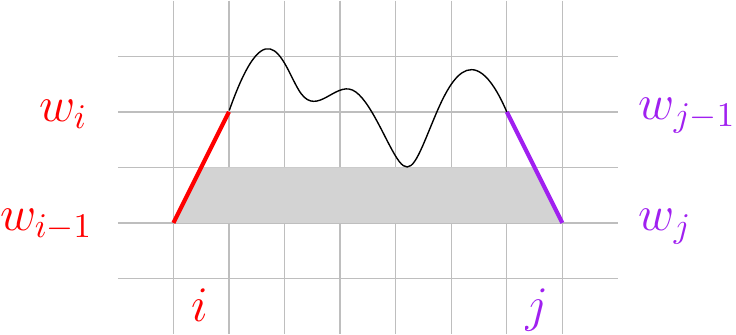}
	\caption{The $i$-th and $j$-th steps of this walk are paired (the open gray area must not be intersected by the walk). A \emph{staggered $\pm 2$-step} is a $\pm2$-step that is not paired with any other $\pm2$-step.}
	\label{fig:staggered}
\end{figure}

\begin{prop}\label{propstag}
For $n\ge 1$ and $m\ge 0$, the number of $n$-step basketball walks from~$0$ to~$1$ that are positive except at the origin and 
that contain exactly~$m$ staggered $\pm2$-steps is equal to the number of $n$-vertex increasing unary-binary trees with associated permutation avoiding~$213$ that have exactly~$m$ unary nodes.
\end{prop}

Proposition~\ref{unaire} gives a closed formula for these numbers. Along the same lines, we obtain a similar expression for basketball excursions instead of $G$-walks.
\begin{prop}\label{propstagA}
For $n\ge 2$, the number of $(n-1)$-step nonnegative basketball walks from~$0$ to~$0$ is\lv{ equal to}
$$\frac1{n}\sum_{k\geq 0} \binom{2n}{k} \binom{n-k-2}{k-1}$$
and, for $k\geq 0$, the term of index~$k$ in the previous sum is the number of walks with exactly $n-1-2k$ 
staggered $\pm2$-steps.
\end{prop}

The remainder of the paper is organized in two sections. In Section~\ref{secbin}, we present and study the bijection between $C$-walks and binary trees. In Section~\ref{secubt}, we investigate the link between $G$-walks and increasing unary-binary trees with associated permutation avoiding~$213$. \sv{In this short version, we present the bijection and prove Theorem~\ref{thmubt}; the proofs of the other statements can all be found in the extended version.}

\lv{\begin{ack}
We thank the ALEA in Europe Young Researcher's Workshop, which enabled us to start working on this problem. We thank Michael Wallner for mentioning this problem to us during this meeting and the other participants that initiated the discussion on this topic. J.B.\ also acknowledges partial support from the GRAAL grant ANR-14-CE25-0014.
\end{ack}}

\section{Basketball walks and binary trees}\label{secbin}

\subsection{Generalities}\label{secgen}

Throughout this paper, we will denote generating functions by capital letters. The class of combinatorial objects counted by a generating function will be denoted by the same letter with a calligraphic font. For instance, the generating function~$A$ will count the objects of the combinatorial class~$\cA$. Moreover, we denote by~$\un$ (resp.~$\cZ$) the combinatorial class whose only element has size~$0$ (resp.\ size~$1$), that is, the combinatorial class corresponding to the generating function~$z\mapsto 1$ (resp.~$z\mapsto z$). Finally, for a class~$\cX$ with no object of size~$0$, we denote by $\sq(\cX)$ the combinatorial class of sequences of elements of~$\cX$, that is $\un+\cX+\cX^2+\ldots$, and by $\sqk(\cX)=\cX^k\sq(\cX)$ the class of sequences of at least~$k$ elements of~$\cX$.

Let us start by introducing the following combinatorial classes of basketball walks\lv{ (recall that this means walks with steps in $\{-2,-1,+1,+2\}$)}:
\begin{itemize}
	\item $\cC$ is the set of walks from~$0$ to~$0$ that visit~$1$ and are positive except at the extremities\lv{ (``$\cC$'' stands for Catalan as these objects are counted by the Catalan numbers)};
	\item $\cA$ is the set of excursions, that is, nonnegative walks from~$0$ to~$0$;
	\item $\cB$ is the set of walks from~$0$ to~$1$ that visit~$0$ and~$1$ only at the extremities.
\end{itemize}
The corresponding generating functions~$C$, $A$ and~$B$ count walks of these classes with a weight~$z$ per step. We also call $X$-walk an element of a class~$\cX$ of walks. 

\lv{
\bigskip
In this section, we will obtain a decomposition grammar (in the spirit of~\cite{labelle1990generalized,duchon2000enumeration}) relating the classes of basketball walks introduced above. By a few manipulations we will deduce from the grammar that the class
$\cC$ satisfies $\cC=(\cZ+\cC)^2$, and is thus a Catalan class.
}

\begin{notn}
It will sometimes be convenient to see walks as sequences of steps. Note that there is a slight abuse in such a notation, as we forget the value of the origin of the walks; anywhere we use this notation, the initial value of the walk will be implicit or irrelevant. We will denote by $\bar\fa\de(-a_n,\ldots,-a_1)$ the reverse of a walk $\fa=(a_1,\ldots,a_n)$ and by $\fa\fb\de(a_1,\ldots,a_n,b_1,\ldots,b_m)$ the concatenation of~$\fa$ and $\fb=(b_1,\ldots,b_m)$.
\end{notn}

We may notice right away that the class~$\cG$ of basketball walks from~$0$ to~$1$ that are positive except at the origin, introduced in Section~\ref{secintro}, satisfies $\cG=\cB\cA$ by splitting any $G$-walk at the first time it reaches height~$1$. Moreover, we can see that $\cB=\cZ \sq(\cZ\cA)$ as follows\lv{ (see Figure~\ref{walkdec}\textbf{b})}. A $B$-walk is either trivial, that is, reduced to a single $+1$-step, or starts with a $+2$-step, forms an $A$-walk between the first and last time it visits~$2$, and then forms a reversed $B$-walk from its last visit of~$2$ to its end. As a result, $\cB=\cZ+\cZ\cA\cB$, and the claim follows. This implies that
\begin{equation}\label{GzA}
\cG=\cZ\cA+\cZ\cA\cG\,,\qquad\text{ so that }\qquad \cG=\squ(\cZ\cA).
\end{equation}
In the remainder of Section~\ref{secbin}, we will focus on~$\cC$ and no longer on~$\cG$. We will come back to~$\cG$ and use~\eqref{GzA} in Section~\ref{secubt}.\lv{ Decomposing $A$-walks, we obtain the following identity\lv{ (see Figure~\ref{walkdec}\textbf{a})}
\begin{equation}\label{eq:idA}
\cA=\un+\cZ\cA\cZ\cA+\cB\cA\cB\cA,
\end{equation} 
which corresponds to the fact that an $A$-walk is either empty, or does not visit~$1$ before its first return to~$0$, or does visit~$1$ before its first return to~$0$. 

\lv{
\begin{figure}[ht]
	\centering\includegraphics[width=.95\linewidth]{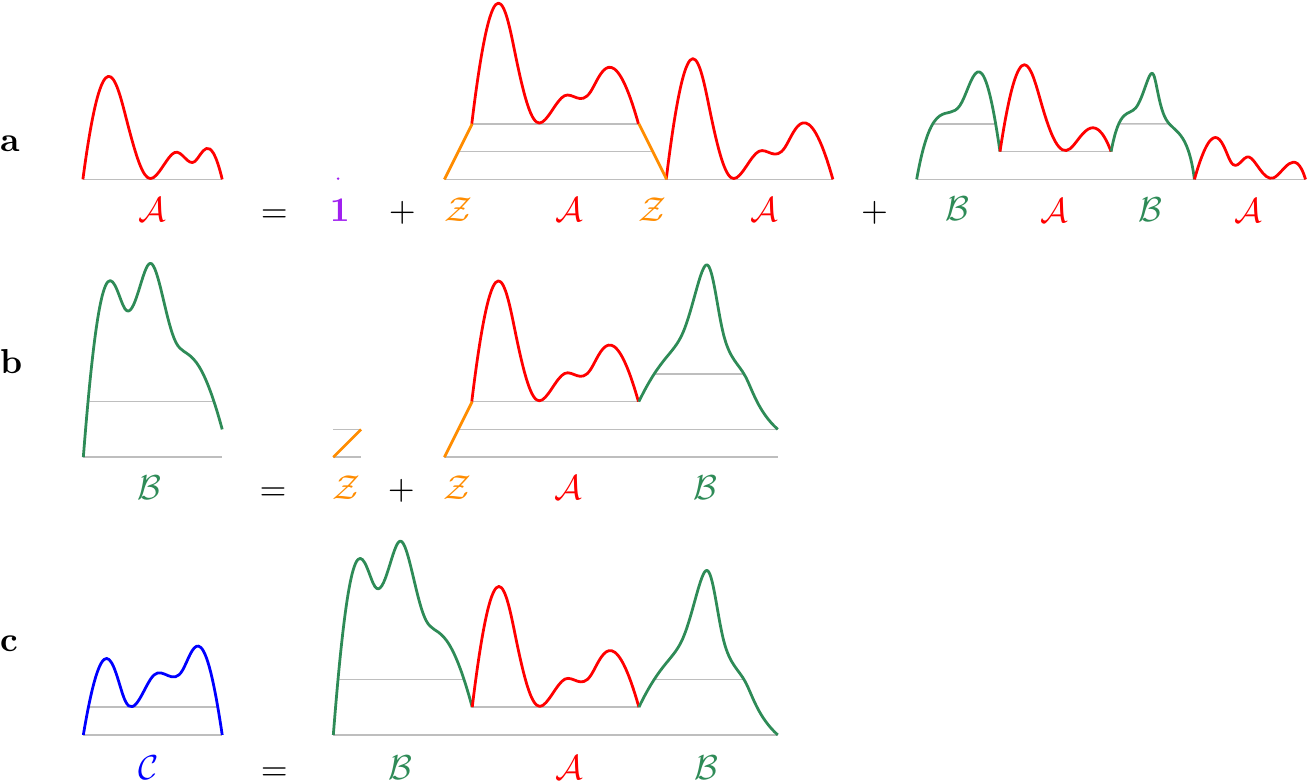}
	\caption{Decomposition of an $A$-walk (\textbf{a}), a $B$-walk (\textbf{b}) and a $C$-walk (\textbf{c}). The horizontal lines are here to emphasize the facts that $A$-walks are nonnegative, $B$-walks only visit~$0$ and~$1$ at their extremities and $C$-walks only visit~$0$ at their extremities. Beware, however, that the walk $(+1)$ is a $B$-walk.}
	\label{walkdec}
\end{figure}
}

Splitting a $C$-walk at the first and last time it visits~$1$, we see that $\cC=\cB\cA\cB$\lv{ (see Figure~\ref{walkdec}\textbf{c})}. We claim that the  three identities $\cB=\cZ \sq(\cZ\cA)$, \eqref{eqC}, and $\cC=\cB\cA\cB$ suffice to show that~$\cC$ satisfies $\cC=(\cZ+\cC)^2$ and thus is the Catalan class whose generating function is given by~\eqref{eqC}. Indeed, if we multiply~\eqref{eq:idA} by~$\cB^2$, we obtain
$$\cB\cA\cB=\cB^2+(\cZ\cA\cB)^2+(\cB\cA\cB)^2\quad\text{ and thus }\qquad\cC=\cB^2+(\cZ\cA\cB)^2+\cC^2.$$
Now observe that, for an arbitrary combinatorial class~$\cR$ with no object of size~$0$, we have
\begin{align*}
\sq(\cR)^2+\big(\cR\sq(\cR)\big)^2&=\sq(\cR)\sq(\cR)+\squ(\cR)\squ(\cR)\\
&=\un\!\cdot\!\un+\squ(\cR)\!\cdot\!\un+\sq(\cR)\squ(\cR)+\squ(\cR)\squ(\cR)\\
&=\un+\squ(\cR)\sq(\cR)+\sq(\cR)\squ(\cR)=\un+2\cR\ \!\sq(\cR)^2,
\end{align*}
To obtain the second equality, we split $\sq(\cR)^2$ into three terms and we recombined the second and forth terms in order to obtain the third equality. As $\cB=\cZ\sq(\cZ\cA)$, the latter identity applied to $\cR=\cZ\cA$ yields
\begin{align*}
\cB^2+(\cZ\cA\cB)^2&=\cZ^2\big(\un+2\cZ\cA\sq(\cZ\cA)^2\big)=\cZ^2+2\cZ\cA\cB^2=\cZ^2+2\cZ\cC,
\end{align*}
so that we obtain $\cC=\cZ^2+2\cZ\cC+\cC^2=(\cZ+\cC)^2$ as claimed. 

\lv{
\bigskip
If we now mark each $\pm 1$-step by~$\cZ_1$ and each $\pm 2$-step by~$\cZ_2$, then it is easy to see that the identities above become
$$\cA=\un+\cZ_2\cA\cZ_2\cA+\cB\cA\cB\cA\,,\qquad \cB=\cZ_1\sq(\cZ_2\cA)\,,\qquad \cC=\cB\cA\cB\,,$$
so that~$\cC$ now satisfies
\begin{align*}
\cC&=\cB^2+\big(\cZ_2\cA\cB\big)^2+\cC^2=\cZ_1^2+2\cZ_2\,\cC+\cC^2.
\end{align*}
In a binary tree, we define a \emph{double leaf} to be a leaf whose sibling is also a leaf and a \emph{simple leaf} to be a leaf whose sibling is a node. The previous identity corresponds to the decomposition grammar of binary trees where each double leaf is marked by~$\cZ_1$ and each simple leaf is marked by~$\cZ_2$. 
}

\bigskip
Note that a simple leaf is either the left child or the right child of its parent; it is called a \emph{left leaf} of a \emph{right leaf} accordingly. In the next section, we present a recursive bijection between~$\cC$ and binary trees that essentially implements the above decomposition,\lv{ (thus mapping $\pm 1$-steps to double leaves and $\pm 2$-steps to simple leaves),}  along with suitable reversal operations in order to further map the two types of simple leaves (left or right) of the binary tree to explicit (parity-dependent) types of $\pm 2$-steps on the corresponding $C$-walk. 
}

\subsection{The bijection}\label{secbij}

We will now rewrite the grammar of the basketball walks we consider in a way that reflects more faithfully the Catalan decomposition. Recall\lv{ from Figure~\ref{walkdec}\textbf{c}} that $\cC=\cB\cA\cB$. This decomposition will be intensively used in what follows.

As a $C$-walk starts and ends at~$0$, is positive in between and has steps of height at most~$2$, it makes sense to focus on the sequence of heights it reaches, restricted to the set $\{1,2\}$. For instance, the sequence corresponding to the walk taking the consecutive values $(0,2,3,4,2,3,1,2,3,5,4,2,\allowbreak 3,1,3,1,2,0)$ is $(2,2,1,2,2,1,1,2)$. We classify $C$-walks into~$4$ classes according to this sequence of values as follows (the boxed values are mandatory\footnote{For instance, a walk in class~\ref{cas2} will visit~$2$ a nonnegative number of times, then~$1$ for the first time, then~$2$ at least once before it visits~$1$ for the second time and then is free to visit~$1$ and~$2$ an arbitrary number of times and in any order.}, each~$x$ can be either~$1$ or~$2$). See Figure~\ref{figbij}.
\begin{enumerate}[label=(\textit{\roman*})]
	\item $\boxd{1}\,$;\label{cas1}
	\item $2\ldots 2\ \boxd{1} \boxd{2}\ 2 \ldots 2 \ \boxd{1}\ x \ldots x\,$;\label{cas2}
	\item $2\ldots 2\ \boxd{1} \boxd{2}\ 2 \ldots 2 \,$;\label{cas3}
	\item $2\ldots 2\ \boxd{2} \boxd{1}$ or $2\ldots 2\ \boxd{1} \boxd{1}\ x \ldots x \,$.\label{cas4}
\end{enumerate}

Let us analyze more closely these different classes and, at the same time, define a function $\Phi:\cC\to(\cC\cup\{\eps\})^2$, where~$\eps$ is a formal symbol. The function~$\Phi$ will be the starting point of the bijection. We refer the reader to Figure~\ref{figbij} for a visual aid.

\sv{
\begin{figure}[ht!]
	\centering\includegraphics[width=.95\linewidth]{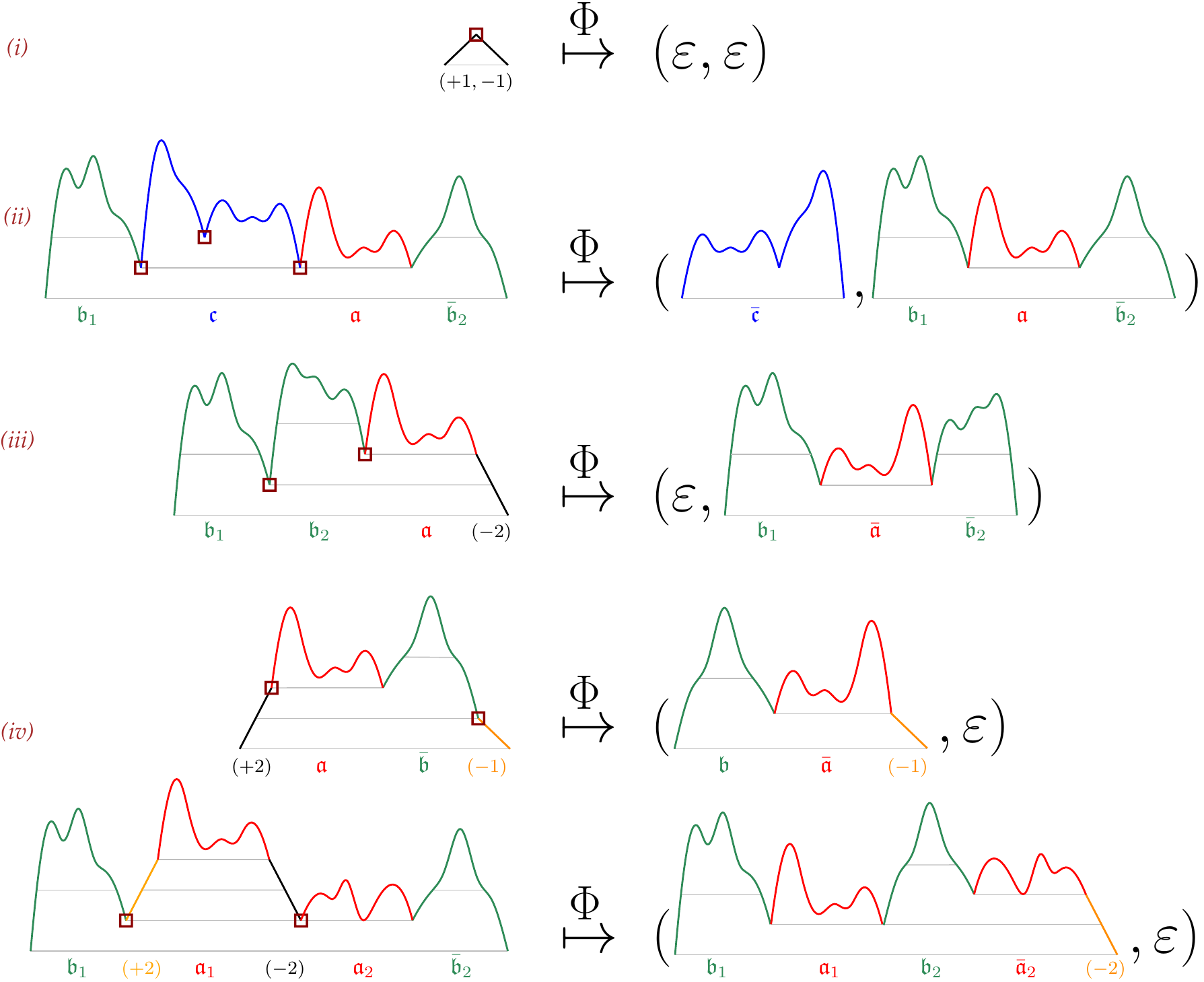}
	\caption{The four different classes and the definition of the function~$\Phi$. Beware of our representation of $B$-walks; recall that $(+1)$ is a $B$-walk. The reversal operation on some $C$-walks and $A$-walks was done in order to preserve the statistics of Proposition~\ref{propstats}.}
	\label{figbij}
\end{figure}
}
\lv{
\begin{figure}[ht!]
	\centering\includegraphics[width=.95\linewidth]{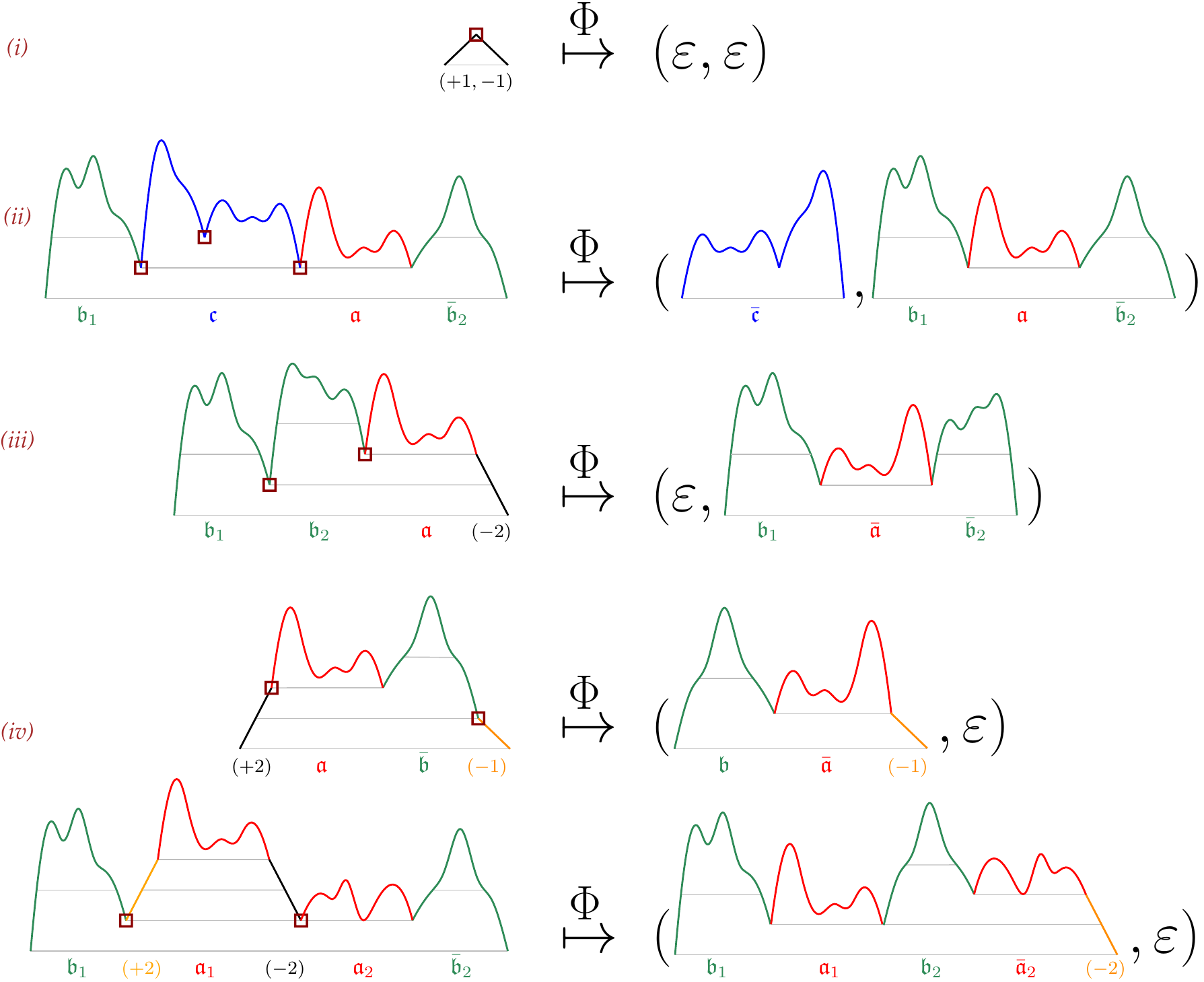}
	\caption{The four different classes and the definition of the function~$\Phi$. The reversal operation on some $C$-walks and $A$-walks was done in order to preserve the statistics of Proposition~\ref{propstats}. For the purposes of Section~\ref{secbij}, an alternate definition of~$\Phi$ without these reversal operations also works. However, the choices of reversals we made will be crucial for Section~\ref{secstat}.}
	\label{figbij}
\end{figure}
}

\paragraph{Class~\ref{cas1}.}
The only walk in class~\ref{cas1} is $(+1,-1)$; we define $\Phi\big((+1,-1)\big)\de(\eps,\eps)$.

\paragraph{Class~\ref{cas2}.}
Such a walk is the concatenation of three subwalks: a $B$-walk $\fb_1$ \big($2\ldots 2\ \boxd{1}$\big), a $C$-walk~$\fc$ \big($\boxd{1} \boxd{2}\ 2 \ldots 2 \ \boxd{1}$\big), and the concatenation of an $A$-walk~$\fa$ and a reversed $B$-walk~$\bar\fb_2$ \big($\boxd{1}\ x \ldots x$\big). We set $\Phi(\fb_1\fc\fa\bar\fb_2)\de(\bar\fc,\fb_1 \fa \bar\fb_2)$.\lv{ (At this point, it might seem odd to have chosen~$\bar\fc$ instead of~$\fc$; this choice will only have bearings in Section~\ref{secstat}. See the caption of Figure~\ref{figbij} for more details.)}

\paragraph{Class~\ref{cas3}.}
Such a walk is composed of a $B$-walk~$\fb_1$ \big($2\ldots 2\ \boxd{1}$\big), another $B$-walk~$\fb_2$ \big($\boxd{1} \boxd{2}$\big), an $A$-walk~$\fa$ and a $-2$-step \big($\boxd{2}\ 2 \ldots 2$\big). We set $\Phi(\fb_1\fb_2\fa(-2))\de(\eps,\fb_1 \bar\fa \bar\fb_2)$.

\paragraph{Class~\ref{cas4}.}
A walk in the first subclass is made of a $2$-step and an $A$-walk~$\fa$ \big($2\ldots 2\ \boxd{2}$\big), then a reversed $B$-walk~$\bar\fb$ and a $-1$-step \big($\boxd{2} \boxd{1}$\big). We set $\Phi\big((+2)\fa\bar\fb(-1)\big)\de\big(\fb\bar\fa(-1),\eps\big)$.

A walk in the second subclass can be decomposed into a $B$-walk~$\fb_1$ \big($2\ldots 2\ \boxd{1}$\big), a $+2$-step, an $A$-walk~$\fa_1$ and a $-2$-step \big($\boxd{1} \boxd{1}$\big), and the concatenation of an $A$-walk~$\fa_2$ and a reversed $B$-walk~$\bar\fb_2$ \big($\boxd{1}\ x \ldots x$\big). We forget the $-2$-step and rearrange the remaining parts as follows: we set $\Phi\big(\fb_1(+2)\fa_1(-2)\fa_2\bar\fb_2\big)\de\big(\fb_1\fa_1\fb_2\bar\fa_2(-2),\eps\big)$.

\lv{Observe that the walks in the two previous subclasses together yield an even $+2$-step or an odd $-2$-step, as well as a $C$-walk whose final step is~$-1$ or~$-2$ depending on the subclass.

\bigskip}
\sv{\medskip}
We recover from the previous classification that $\cC=\cZ^2+\cC^2+\cZ\cC+\cZ\cC=(\cZ+\cC)^2$. \lv{To make the parallel with Section~\ref{secgen}, the classes we considered correspond to decomposing~$\cC$ as 
\begin{align*}
\cC&=\underbrace{\vphantom{\squ{}}\cZ^2}_{\ref{cas1}}+\underbrace{\cZ\!\cdot\!\cZ\squ(\cZ\cA)}_{\ref{cas4}.1}+\underbrace{\cZ\squ(\cZ\cA)\!\cdot\!\cZ\sq(\cZ\cA)}_{\ref{cas3}}
		+\underbrace{\cZ^2\cA^2\big(\cZ\sq(\cZ\cA)\big)^2}_{\ref{cas4}.2}+\underbrace{\vphantom{\squ{}}\cC^2}_{\ref{cas2}}. 
\end{align*}
The mapping~$\Phi$ amounts to rearranging~\ref{cas3} as $\cZ\cZ\cA\sq(\cZ\cA)\cZ\sq(\cZ\cA)=\cZ\cB\cA\cB=\cZ\cC$ and~\ref{cas4} as $\cZ^3\cA\sq(\cZ\cA)\big(\un+\cZ\cA\sq(\cZ\cA)\big)=\cZ^3\cA\sq(\cZ\cA)^2=\cZ\cB\cA\cB=\cZ\cC$. 
}Now, the previous analysis can be turned into a (recursive) bijection between $C$-walks and binary trees. We take a $C$-walk and recursively build the corresponding tree as follows. We start from the one-vertex tree and assign to its unique node the $C$-walk. Then, recursively, we do the following for every node with an assigned walk. We denote by $(\fp_1,\fp_2)$ the value of~$\Phi$ on the assigned walk and glue to the considered node two new vertices. If $\fp_1=\eps$, then we let the left vertex as a leaf; otherwise we assign to it the walk~$\fp_1$. We proceed similarly with the right vertex and~$\fp_2$. See Figure~\ref{exbij} for an example.

\sv{
\begin{figure}[ht!]
	\centering\includegraphics[width=.5\linewidth]{exbij_sv}
	\caption{The first stages of the bijection between a $14$-step $C$-walk and a $14$-leaf binary tree. The original $C$-walk is represented at the root of the tree. We represented by purple crosses the odd $+2$-steps and even $-2$-steps of the walks; observe that these are preserved at every stage. The other $\pm2$-steps satisfy a similar property.}
	\label{exbij}
\end{figure}
}
\lv{
\begin{figure}[ht!]
	\centering\includegraphics[width=.7\linewidth]{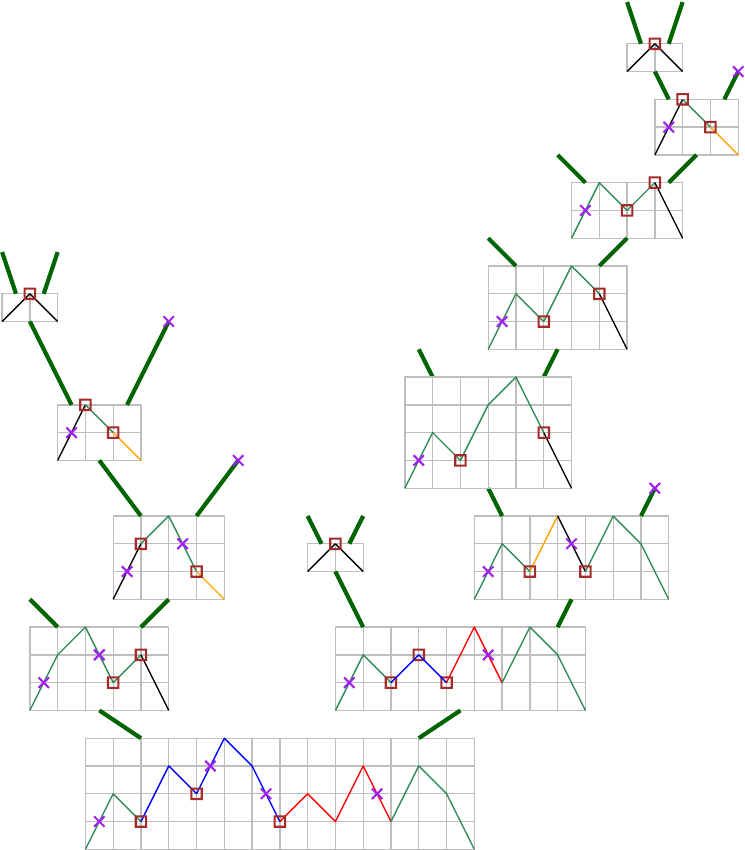}
	\caption{The bijection between a $14$-step $C$-walk and a $14$-leaf binary tree. The original $C$-walk is represented at the root of the tree and the walks assigned to the nodes of the tree are represented. We used the same color scheme as on the left of Figure~\ref{figbij} and marked by purple crosses the even $+2$-steps and odd $-2$-steps of the walks; observe that these are preserved at every stage except when a right leaf is created (these are also marked by purple crosses on the figure), in which case one such step disappears. The other $\pm2$-steps satisfy a similar property with left leaves and the $\pm1$-steps correspond to double leaves.}
	\label{exbij}
\end{figure}
}

Denoting by~$\Phi_1$ and~$\Phi_2$ the coordinate projections of~$\Phi$ and by $|\cdot|$ the number of steps in a walk, it is easy to check from the definition of~$\Phi$ that $|\fp|=|\Phi_1(\fp)|+|\Phi_2(\fp)|$, if we use the convention that $|\eps|\de 1$. As a consequence, we see that the number of leaves in the binary tree corresponding to a walk~$\fp$ is~$|\fp|$. We thus obtain, for each $n\ge 1$, a bijection between $C$-walks of length $n+1$ and binary trees with~$n+1$ leaves, which are counted by the $n$-th Catalan number.

\subsection{Matched statistics}\label{secstat}

Let us now see which quantities are matched by our bijection. \sv{We first need to introduce some terminology. A leaf of a tree is called \emph{double leaf} if its sibling is also a leaf, \emph{left leaf} if it is a left child and its sibling is not a leaf, and \emph{right leaf} if it is a right child and its sibling is not a leaf.}

\begin{prop}\label{propmatch}
We consider a $C$-walk and \sv{the corresponding}\lv{a} binary tree\lv{ corresponding to each other through the bijection}. Then
\begin{itemize}
	\item the number of $\pm 1$-steps of the walk is equal to the number of double leaves of the tree;
	\item the sum of the numbers of odd $+2$-steps and even $-2$-steps of the walk is equal to the number of left leaves of the tree;
	\item the sum of the numbers of odd $-2$-steps and even $+2$-steps of the walk is equal to the number of right leaves of the tree.
\end{itemize}
\end{prop}

\lv{
\begin{pre}
We proceed recursively. Let $d(\fp)$ denote the number of $\pm 1$-steps of a walk~$\fp$. As all transformations involved in the definition of~$\Phi$ can only change a $k$-step into a $k$-step or a $-k$-step, we see that, for $\fp\in\cC\setminus\{(+1,-1)\}$, we have $d(\fp)=d(\Phi_1(\fp))+d(\Phi_2(\fp))$, with the convention that $d(\eps)\de 0$. Moreover, the walk $(+1,-1)$ has two $\pm 1$-steps and creates two double leaves, whereas any other walk does not create any double leaves. The first statement of the proposition follows.

Now remember that both $A$-walks and $C$-walks start and end at height~$0$, and that $B$-walks start at height~$0$ and end at height~$1$. Moreover, denoting by $\ell(\fp)$ the sum of the numbers of odd $+2$-steps and even $-2$-steps of a walk~$\fp$, we have $\ell(\fp+2k)=\ell(\bar\fp+2k+1)=\ell(\fp)$ for any integer~$k$ (where $\fp+p$ denotes the path with same step sequence as~$\fp$ and shifted by~$p$). In other words, the quantity~$\ell$ remains unchanged after a shift of even height or after a shift of odd height and a reversal. Bearing this in mind, observe that~$\Phi$ was chosen in such a way that, in class~\ref{cas1}, $\ell(\fp)=0$, in class~\ref{cas2}, $\ell(\fp)=\ell(\Phi_1(\fp))+\ell(\Phi_2(\fp))$, in class~\ref{cas3}, $\ell(\fp)=1+\ell(\Phi_2(\fp))$, and in class~\ref{cas4}, $\ell(\fp)=\ell(\Phi_1(\fp))$. The second statement follows from the fact that walks of class~\ref{cas3} create exactly one left leaf whereas walks of other classes do not create any.

The third statement follows by a similar reasoning or by noticing that the sum of the three considered statistics pertaining to a walk is equal to its length and that the sum of the three considered statistics pertaining to a tree is equal to its number of leaves, as we already observed that the length of a walk was matched by the bijection to the number of leaves of the corresponding tree.
\end{pre}
}

\sv{The proof of this proposition is quite straightforward and }Proposition~\ref{propstats} is a direct consequence of \sv{it}\lv{Proposition~\ref{propmatch}}, as the desired number is then equal to the number of binary trees with~$d$ double leaves, $\ell$ left leaves and~$r$ right leaves. \sv{We refer the reader to the extended version of this paper for the details.}\lv{This number can easily be obtained from~\cite[Theorem~4]{mckenzie2000distributions} (which gives the number of non-embedded rooted binary trees with~$n$ labeled leaves, and~$k$ pairs of double leaves, called \emph{cherries}), or from counting unary-binary trees; we provide here a short line of proof for the sake of self-containedness (note that one can also obtain the expression through the Lagrange inversion formula, as in the proof of Proposition~\ref{unaire}). 

\begin{pre}[Proof of Proposition~\ref{propstats}]
A binary tree with~$2d$ double leaves, $\ell$ left leaves and~$r$ right leaves can be obtained as follows. We start from a binary tree with a root of degree~$1$ and~$d$ leaves and we add $\ell+r$ extra vertices on its $2d-1$ edges. We graft two new edges to each of the~$d$ original leaves. Among the added vertices, we select~$\ell$ vertices to the left of which we graft an edge and we graft an edge to the right of the unselected vertices. As this operation is clearly bijective, we obtain that the desired number is equal to the $(d-1)$-th Catalan number, times the number of ways to put $\ell+r$ extra vertices on $2d-1$ edges, times the number of ways to select~$\ell$ vertices among $\ell+r$.
\end{pre}
}

\sv{
\begin{rem}
From Proposition~\ref{propmatch}, our bijection specializes into a bijection between Dyck walks and binary trees with only double leaves. The latter are in direct bijection with binary tree with half as many leaves. In fact, it can be checked that, in this case, our bijection gives the classical encoding of binary trees by Dyck walks, up to mirroring the left subtree at each step (because of the reversal of~$\fc$ in class~\ref{cas2}).
\end{rem}
}
\lv{
\subsection{Walks without large steps}

Let us restrict our attention to $C$-walks without $\pm 2$-steps. Such walks are in direct bijection with $A$-walks without $\pm 2$-steps by removing the first and last steps (recall that a $C$-walk may not visit~$0$ more than twice, whereas an $A$-walk may) and the latter are exactly Dyck walks. If~$\fp$ is such a $C$-walk, it can only be in the classes~\ref{cas1} or~\ref{cas2}. Moreover, the only $B$-walk without $\pm 2$-steps is the trivial walk $(+1)$, so that, if~$\fp$ is of class~\ref{cas2}, it is of the form $(+1,+1)\fa_1 (-1) \fa_2 (-1)$ where~$\fa_1$ and~$\fa_2$ are $A$-walks without $\pm 2$-steps. As a consequence, we see that $\Phi\big((+1,+1)\fa_1 (-1) \fa_2 (-1)\big)=\big((+1)\bar\fa_1(-1),(+1) \fa_2 (-1)\big)$ and our bijection is in this case a slight modification of the classical bijection between Dyck walks and binary trees.

Recall that the classical bijection may be constructed in a similar way as ours with a function~$\Psi$ defined on nonempty Dyck walks by $\Psi\big((+1)\fa_1 (-1) \fa_2\big)\de(\fa_1,\fa_2)$. Then, from a Dyck walk, we recursively construct the tree by starting with a vertex tree and assigning to its vertex the walk. Then, we split every vertex with an assigned nonempty walk~$\fa$ into two new vertices to which we assign the coordinates of~$\Psi(\fa)$. Through this bijection, the length of the Dyck walk is twice the number of leaves minus~$2$.

Plainly, our bijection gives the same result with two double leaves grafted onto each leaf (as the empty Dyck walk yields a leaf in the classical bijection and the corresponding $C$-walk $(+1,-1)$ yields two double leaves in our construction) and where each left subtree is mirrored at each step (because of the reversal of~$\fa_1$).

Summing up, we see that the restriction of our bijection to $C$-walks without $\pm 2$-steps gives a bijection between such walks and binary trees, which is a slight alteration of the classical encoding of binary trees by Dyck walks. Note that, here, the length of the $C$-walk is twice the number of leaves of the corresponding binary tree (equivalently, it is the number of leaves of the tree obtained by adding two double leaves to each original leaf).
}

\section{Link with increasing unary-binary trees}\label{secubt}

\lv{
\subsection{Proof of Theorem~\ref{thmubt}}
}

Let us now investigate the connection between basketball walks and increasing unary-binary trees with associated permutation avoiding~$213$. We start by giving\lv{ a more convenient characterization of the permutations allowed on a given (unlabeled) unary-binary tree. We say that a permutation $\si=(\si_1,\ldots,\si_n)$ is \emph{valid} for a unary-binary tree if the labeled tree obtained by labeling the vertices by~$\si_1$, \ldots, $\si_n$ in (left-to-right) breadth-first search order is an increasing unary-binary tree with associated permutation avoiding~$213$. In other words, the labeled tree is increasing and~$\si$ avoids~$213$.

\begin{lem}\label{lembits}
We consider an $n$-vertex unary-binary tree and denote by $\sN\subseteq\{1,2,\ldots,n\}$ the set of indices of its nodes when reading its vertices in breadth-first search order. The permutation~$\si$ is valid for the tree if and only if it avoids~$213$ and, for all $i\in \sN$, $\si_i=\min_{i\le j\le n} \si_j$; in other words, the elements of~$\sN$ are indices of right-to-left minimums.
\end{lem}

\begin{rem}
This observation was made in the very particular case of complete unary-binary trees, for which the vertices have a very simple structure as they are arranged by nodes first and leaves last; see \cite[Theorem~4]{LPRS16heaps}.
\end{rem}

\begin{pre}
Let us denote by~$v_1$, \ldots, $v_n$ the vertices of the tree read in breadth-first search order. By definition, $\si$ is valid for the tree if and only if it avoids~$213$ and, for all $i\in \sN$ and all $j\in\{1,\ldots,n\}$ such that~$v_j$ is a child of~$v_i$, we have $\si_i<\si_j$. By definition of the breadth-first search order, if~$i$ and~$j$ are as above, then $i<j$. As a result, a permutation for which the elements of~$\sN$ are indices of right-to-left minimums is immediately valid.

Conversely, let us take a valid permutation~$\si$ and an index $i\in \sN$; it is sufficient to show that~$i$ is the index of a right-to-left minimum, that is, $\si_i<\si_k$ for all $k\in\{i+1,\ldots,n\}$. Let $j\in\{1,\ldots,n\}$ be the largest integer such that~$v_j$ is a child of~$v_i$. First, $\si_i<\si_j$ as the labeled tree is increasing. Second, we must have $\si_i<\si_k$ for all $k\in\{i+1,\ldots,j-1\}$ as $\si_i<\si_j$ and $(\si_i,\si_k,\si_j)$ must avoid~$213$. Last, if $k\in\{j+1,\ldots,n\}$, then~$v_{k}$ must have an ancestor~$v_{k'}$ for some $k'\in\{i+1,\ldots,j\}$ and we see that $\si_k>\si_{k'}>\si_i$ by the above observation.
\end{pre}

\begin{lem}\label{lemavoid}
For $n\geq 1$, let $\cP_n$ be the set of $213$-avoiding permutations of size~$n$ and, for each nonempty subset $\mathscr E\subseteq\{1,\ldots,n-1\}$, let $\cP_n^{(\mathscr E)}\de\{\si\in\cP_n\,:\, \text{each $a\in \mathscr E$ is the index of a right-to-left minimum of }\si\}$. Then, for such a subset $\mathscr E=\{a_1<\ldots<a_k\}$, we have
$$\cP_n^{(\mathscr E)}\simeq \cP_{a_1}\times\cP_{a_2-a_1}\times\cdots\times\cP_{a_k-a_{k-1}}\times\cP_{n-a_k}.$$ 
\end{lem}
\begin{pre}
The mapping from $\si\in\cP_n^{(\mathscr E)}$ to $(\si^{(0)},\ldots,\si^{(k)})\in\cP_{a_1}\times\cP_{a_2-a_1}\times\cdots\times\cP_{a_k-a_{k-1}}\times\cP_{n-a_k}$ is very simple: for $i\in\{0,\ldots,k\}$, we let $\sigma^{(i)}$ be the renormalized permutation induced by $\sigma$ on the index-set
$\{a_{i}+1,\ldots,a_{i+1}\}$ (with the convention $a_0=0$ and $a_{k+1}=n$).

\begin{figure}[ht]
	\centering\includegraphics[width=.95\linewidth]{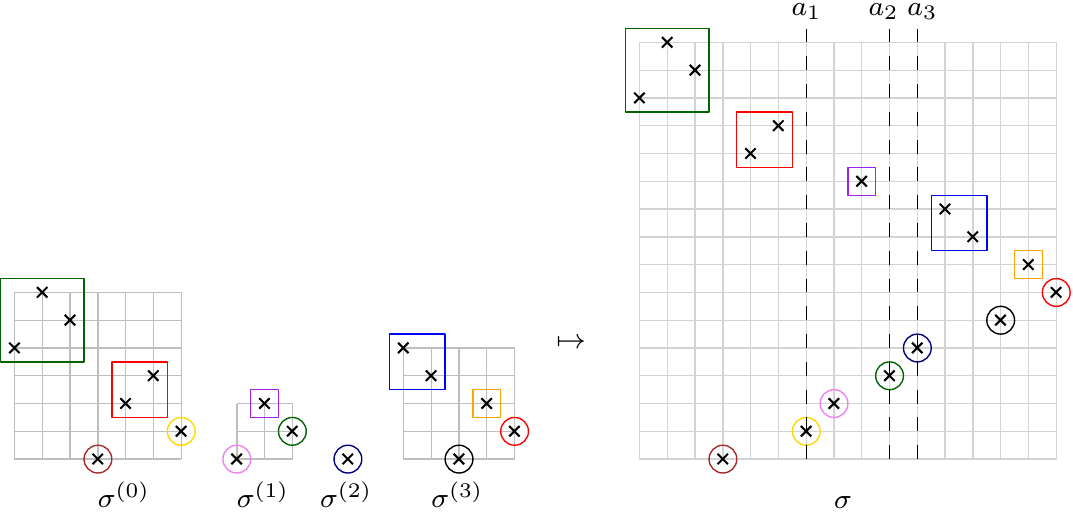}
	\caption{From a $(k+1)$-sequence $(\sigma^{(0)},\ldots,\sigma^{(k)})$ of nonempty $213$-avoiding permutations ($k=3$ in the example) to a $213$-avoiding permutation~$\sigma$ of size $|\sigma^{(0)}|+\cdots+|\sigma^{(k)}|$. If we let $a_i\de|\sigma^{(0)}|+\cdots+|\sigma^{(i-1)}|$ for each $i\in\{1,\ldots,k\}$, then~$a_i$ is the index of a right-to-left minimum of~$\sigma$. On the figure, the right-to-left minimums are circled and the (nonempty) blocks are framed.}
	\label{decomp_perm}
\end{figure}    

To inverse the mapping we rely on the crucial property that a permutation $\sigma\in\cP_n$ with~$p$ right-to-left minimums is of the form 
$\sigma=(\si^{\langle 0\rangle},1,\si^{\langle 1\rangle},2,\si^{\langle 2\rangle},3,\ldots,\si^{\langle p-1\rangle},p)$ where~$1$, \ldots, $p$ are the  right-to-left minimums and~$\si^{\langle 0\rangle}$, \ldots, $\si^{\langle p-1\rangle}$ are (possibly empty) subpermutations of $\{p+1,\ldots,n\}$ avoiding~$213$ with disjoint ranges arranged in decreasing order (that is, $\min \si^{\langle i\rangle} > \max\si^{\langle j\rangle}$ whenever $i<j$ and $\si^{\langle i\rangle}$, $\si^{\langle j\rangle}\neq\emptyset$); these subpermutations are called the \emph{blocks} of $\si$. 

Starting from $(\si^{(0)},\ldots,\si^{(k)})\in\cP_{a_1}\times\cP_{a_2-a_1}\times\cdots\times\cP_{a_k-a_{k-1}}\times\cP_{n-a_k}$, we get to $\sigma\in\cP_n^{(\mathscr E)}$ as follows: we concatenate $\si^{(0)}$, \ldots, $\si^{(k)}$ and renormalize the values in the unique way so that the right-to-left minimums of~$\sigma$ are those of~$\si^{(0)}$, \ldots, $\si^{(k)}$ in increasing order (with consecutive values) and the blocks of~$\sigma$ are those of~$\si^{(0)}$, \ldots, $\si^{(k)}$, with ranges in decreasing order; see Figure~\ref{decomp_perm} for an example. Note that, for each $i\in\{0,\ldots,k-1\}$, since the last index in~$\si^{(i)}$ is necessarily the index of a right-to-left minimum for $\si^{(i)}$, the corresponding index~$a_{i+1}$ is the index of a right-to-left minimum for~$\sigma$.
\end{pre}

Using the previous two lemmas, we may now give} a more convenient encoding of \sv{these objects}\lv{unary-binary trees with associated permutation avoiding~$213$}. Recall that a \emph{Motzkin walk} is a walk with step-set $\{-1, 0, +1\}$. We call \emph{descending sequence} of a Motzkin walk a maximal run of~$-1$'s, that is, a sequence of consecutive steps of the form $(-1,\ldots, -1)$ that is either initial or preceded by a $+1$- or $0$-step and is either final or succeeded by a $+1$- or $0$-step. We finally say that a Motzkin walk is \emph{decorated} if each of its descending sequences is marked with a permutation avoiding~$213$ whose size is the length of the descending sequence (its number of~$-1$'s) plus one.

\begin{lem}\label{lemMotz}
For $n\ge 1$, $n$-vertex increasing unary-binary trees with associated permutation avoiding~$213$ are bijectively encoded by $n$-step decorated Motzkin walks from~$0$ to~$1$ that are positive except at the origin.
\end{lem}

\sv{We refer the reader to the extended version for the proof; }Figure~\ref{exmotz} gives an example\lv{ of this encoding}.

\sv{
\begin{figure}[ht]
	\centering\includegraphics[width=.95\linewidth]{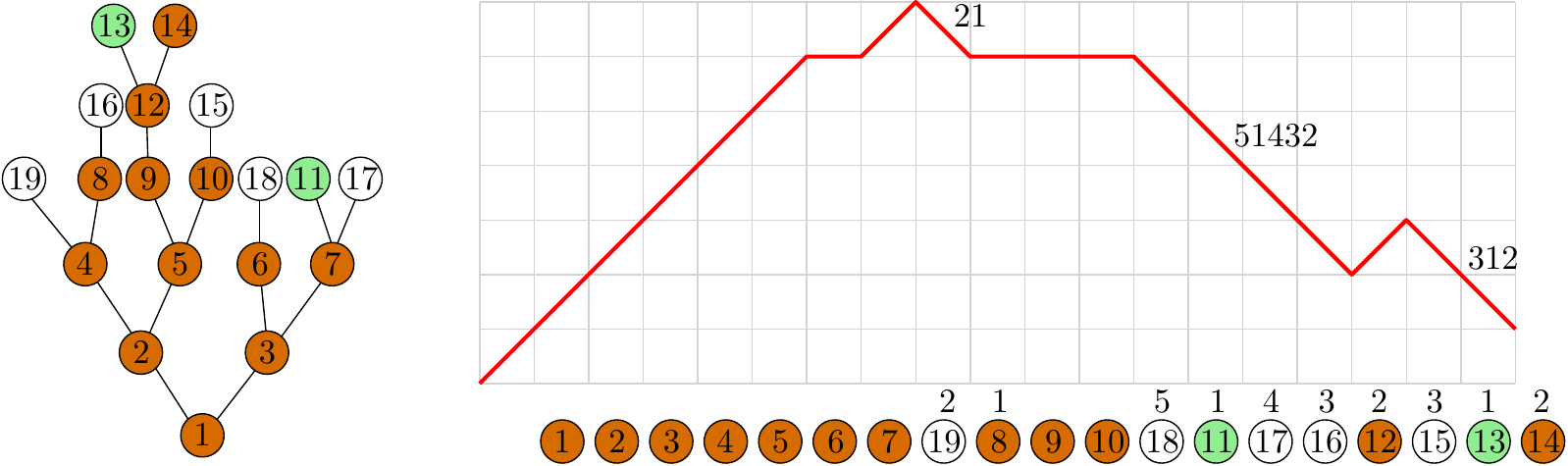}
	\caption{\textbf{Left:} an increasing unary-binary tree with associated permutation avoiding~$213$. \textbf{Right:} its encoding as a decorated Motzkin walk. The nodes and the last leaf are colored in brown and the other right-to-left minimums are colored in green.}
	\label{exmotz}
\end{figure}
}\lv{
\begin{figure}[ht]
	\centering\includegraphics[width=.95\linewidth]{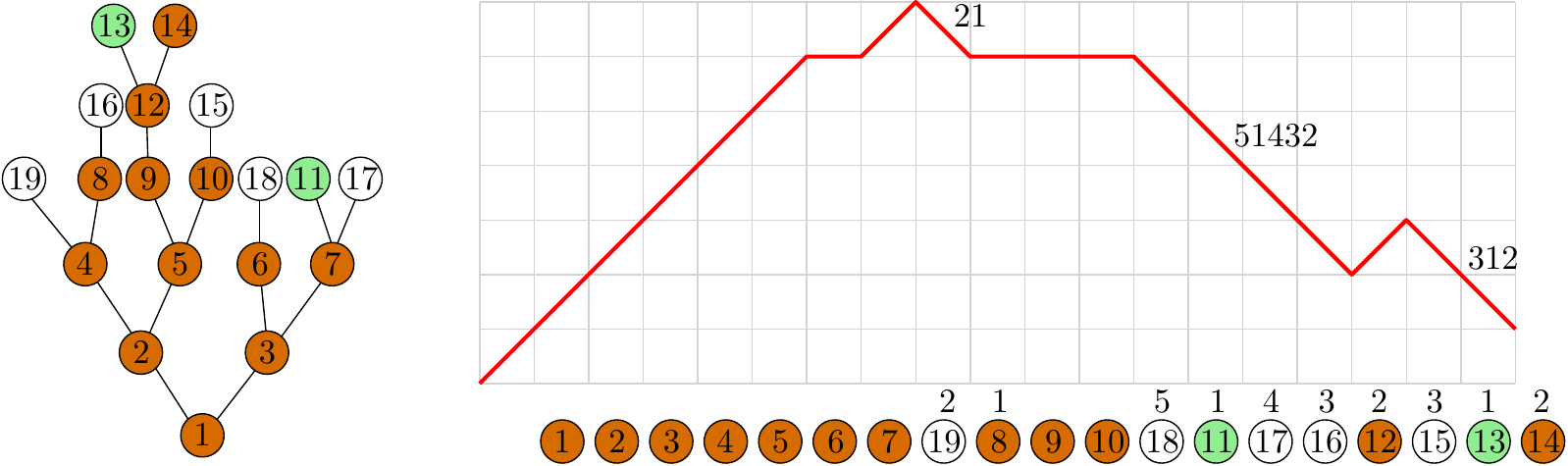}
	\caption{\textbf{Left:} an increasing unary-binary tree with associated permutation avoiding~$213$. \textbf{Right:} its encoding as a decorated Motzkin walk. The nodes and the last leaf are colored brown and the other right-to-left minimums are colored green. Up to a slight modification at its extremities, the Motzkin walk is the classical {\L}ukasiewicz encoding of the tree. The decorations are the renormalized permutations corresponding to the labels of the maximal sequences of leaves when reading the vertices of the tree in breadth-first search order. The full permutation is recovered from the decorations by the mapping of Figure~\ref{decomp_perm}, after assigning the permutation~$1$ to each $0$- or (non-initial) $+1$-step that does not directly follows a descending sequence.}
	\label{exmotz}
\end{figure}

\begin{pre}
Let~$\mathbf{T}$ be a unary-binary tree with~$n$ vertices and associated permutation~$\si$ avoiding the pattern~$213$. 
As above, we denote by~$v_1$, \ldots, $v_n$ the vertices of~$\mathbf{T}$ in the (left to right) breadth-first search order. 
We classically encode~$\mathbf{T}$ as a Motzkin walk as follows:
if we denote by $c(v)\in\{0,1,2\}$ the number of children of the vertex~$v$, 
the walk~$\fp$ associated with~$\mathbf{T}$ starts at zero and follows the step sequence $(+1,c(v_1)-1,c(v_2)-1,\ldots,c(v_{n-1})-1)$. 
The usual description\footnote{Note that there is also a classical encoding using the depth-first search order instead of the breadth-first search order. Both constructions give bijections between the same sets and the proofs are similar.} 
follows a slightly different step sequence; namely the first~$+1$ is omitted and a final step $c(v_n)-1=-1$ is added. 
This only changes the fact that we obtain a positive walk from~$0$ to~$1$ instead of a nonnegative walk from~$0$ to~$-1$.
%In order to see that this is a bijective encoding, observe that the height at time~$i$ of the walk we defined is equal to one plus the number of leaves of the $1$-neighborhood of the subtree with vertices~$v_1$, \ldots, $v_i$, that is, the subtree with vertex-set $\{v_1,\ldots,v_i\}\cup\{v\,:\,\exists j\in\{1,\ldots,i\}, v\text{ is a child of } v_j\}$. 

Let us now take into account the labels of the tree. 
Let $\sN=\{a_1<\ldots<a_k\}$ be the set of indices 
of the nodes of~$\mathbf{T}$ (note that $a_1=1$).
According to Lemma~\ref{lembits}, 
each $a\in\sN$ is the index of a right-to-left minimum of~$\si$. 
Hence, $\si\in\cP_n^{(\sN)}$ and, by Lemma~\ref{lemavoid}, 
$\si$ can be identified with the $(k+1)$-tuple $(\si^{(0)},\ldots,\si^{(k)})$, where~$\si^{(i)}$ is the renormalized
permutation induced by~$\si$ on the index-set $\{a_{i}+1,\ldots,a_{i+1}\}$
(with the convention $a_0=0$ and $a_{k+1}=n$). 

In the Motzkin walk~$\fp$, the initial $+1$-step corresponds to the last leaf of~$\mathbf{T}$, the other $+1$- and $0$-steps correspond to the nodes of~$\mathbf{T}$ and the $-1$-steps correspond to the leaves of~$\mathbf{T}$ except for the last one. This entails that, for each $i\in\{0, \ldots, k\}$ such that $a_{i+1}-a_i\ge 2$, the steps of~$\fp$ between indices $a_i+1$ and $a_{i+1}-1$ form a descending sequence of length $a_{i+1}-a_i-1$. To the descending sequence corresponding to such an~$i$, we assign the permutation~$\sigma^{(i)}$ (of size $a_{i+1}-a_i$). We thus obtain a decorated Motzkin walk that encodes the tree~$\mathbf{T}$ with no loss of information, as $\si^{(i)}=(1)$ whenever $a_{i+1}-a_i=1$. See Figure~\ref{exmotz} for an example.
\end{pre}
}

We now introduce two new generating functions of walks, counted with weight~$z$ per step:
\begin{itemize}
	\item $T$ counts decorated Motzkin walks from~$0$ to $1$ that are positive except at the origin;
	\item $M$ counts decorated Motzkin walks from~$0$ to $1$ that are positive except at the origin and whose last $0$- or $+1$-step is a $+1$-step.\lv{ In other words, $M$ counts the trivial walk $(+1)$ and $T$-walks ending with a $+1$-step followed by a descending sequence.} 
\end{itemize}

\begin{pre}[Proof of Theorem~\ref{thmubt}]
In order to conclude, it suffices to show that $T=G$. We will first show that $\cT=\squ(\cM)$ and then that $M=zA$. This will be sufficient as we already noticed in Equation~\eqref{GzA} that $\cG=\squ(\cZ\cA)$.

\medskip
\noindent\textbf{First step:} We claim that $\cT=\squ(\cM)$. Indeed, a $T$-walk is either an $M$-walk or a $T$-walk whose last $0$- or $+1$-step is a $0$-step. In the latter case, let us denote by $k\ge 0$ the number of $-1$-steps after the last $0$-step (which is necessarily at height $k+1$), and by~$n$ the length of the walk. We do some kind of last passage decomposition at height~$k$: the walk is composed of three parts as follows. Between time~$0$ and the last time the walk hits height~$k$ before the last $0$-step, we have a decorated Motzkin walk from~$0$ to~$k$ that stays positive except at its origin; then until time $n-k-1$, we have a $T$-walk; and, finally, we have a decorated walk of the form $(0,-1,\ldots,-1)$ with $k+1$ steps. See Figure~\ref{TseqM}. We change the $0$-step of the last part into a $+1$-step and we concatenate the first part with it in order to obtain an $M$-walk. Note that this construction also works for $k=0$, in which case the first part (colored purple on Figure~\ref{TseqM}) is the empty walk and the outcome of the above operation is the pair consisting of the original walk with last step removed, together with the trivial $M$-walk~$(+1)$. As a result, we see that a $T$-walk whose last $0$- or $+1$-step is a $0$-step can be bijectively decomposed into a pair made of a $T$-walk and an $M$-walk. This yields $\cT=\cM+\cT\cM$, so that $\cT=\squ(\cM)$ as claimed.

\begin{figure}[ht]
	\centering\includegraphics[width=.95\linewidth]{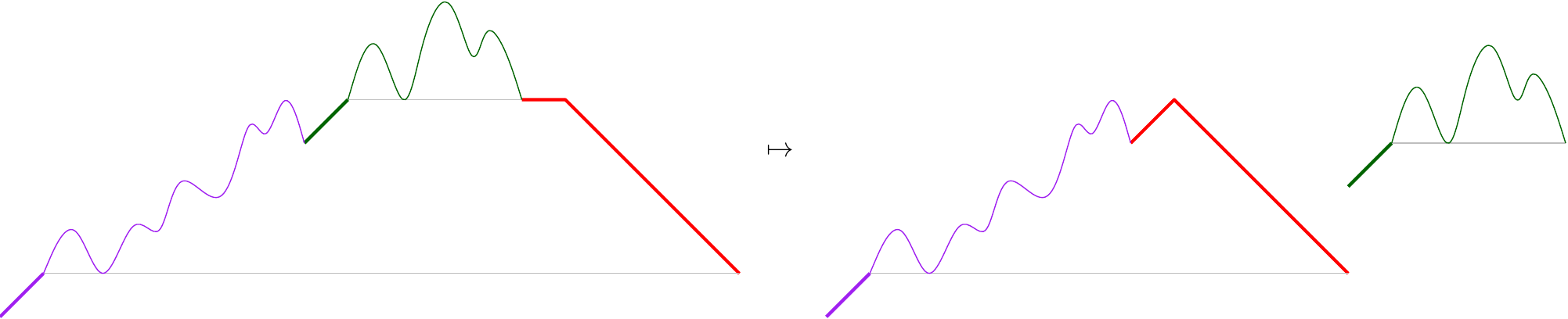}
	\caption{Proof that $T$-walks whose last $0$- or $+1$-step is a $0$-step are counted by $TM$. We extract a $T$-walk (in green) and what remains makes an $M$-walk by changing the last $0$-step into a $+1$-step.}
	\label{TseqM}
\end{figure}

\bigskip
\noindent\textbf{Second step:} Let us now show that $M=zA$. Recall from Section~\ref{secbij} the definitions of~$A$ and~$B$. \sv{We consider the first time an $A$-walk returns to~$0$. An $A$-walk is either empty, or hits~$1$ on its first excursion, or does not hit~$1$ on its first excursion. If it hits~$1$, it is of the form $\fb_1\fa_1\bar\fb_2\fa_2$ where the $\fa_i$'s are $A$-walks and the $\fb_i$'s are $B$-walks, if it does not, it is of the form $(+2)\fa_1(-2)\fa_2$ where the $\fa_i$'s are $A$-walks. This translates into the equation $A=1+BABA+zAzA$. Using~\eqref{GzA}, we can rewrite this equation as}\lv{Recall from the discussion before~\eqref{GzA} that $B=z/(1-zA)$ and from~\eqref{eq:idA} that $A=1+(zA)^2+(BA)^2$. As a result,}
$$zA=z\left(1+(zA)^2+\frac{(zA)^2}{(1-zA)^2}\right).$$

We now turn to~$M$. We consider an $M$-walk and, for the sake of convenience, we suppose for the time being that it is not the trivial walk $(+1)$. Let $k\ge 1$ denote the length of its last descending sequence and~$\si$ denote the permutation of size $k+1$ associated with it. We use the standard Catalan decomposition and split~$\si$ into two permutations avoiding~$213$, $\si'$ and~$\si''$, of respective sizes~$i$ and $k-i$. This means that the first right-to-left minimum of~$\si$ (which is necessarily equal to~$1$) has index $i+1$ and that~$\si'$ and~$\si''$ correspond to $(\si_1,\ldots, \si_{i})$ and $(\si_{i+2},\ldots, \si_{k+1})$, up to relabeling. Let us suppose furthermore that $1\le i \le k-1$ for the moment. We split the $M$-walk at the last time it hits~$i$ before the last $+1$-step and at time $n-k-1$ (that is, before the last $+1$-step); see Figure~\ref{decompM}. We concatenate the first part with the $i$-sequence $(0,-1,\ldots,-1)$ and we associate with the last descending sequence the permutation~$\si'$. Similarly, we add to the second part a $0$-step and a descending sequence of size $k-i-1$ with which we associate the permutation~$\si''$. We thus decomposed our original walk into a single step counted by~$z$ and a pair of $T$-walks whose last $0$- or $+1$-steps are both $0$-steps, which we saw in the first step above to be counted by~$TM$. Now, if $i=k$, we do the same construction for the first walk and obtain an empty second walk. If $i=0$, we obtain an empty first walk and a second walk counted by~$TM$. Finally, if the original walk is trivial, it is counted by~$z$. Putting all this together, we obtain $M=z\,(1+TM)^2$, which can be rewritten as
$$M=z\left(1+\frac{M^2}{1-M}\right)^2=z\left(1-M+\frac{M}{1-M}\right)^2=z\left(1+M^2+\frac{M^2}{(1-M)^2}\right).$$
Since~$zA$ and~$M$ satisfy the same Lagrangian equation of the form $X=z\varphi(X)$, they must be equal. This concludes the proof.
\end{pre}

\begin{figure}[ht]
	\centering\includegraphics[width=.95\linewidth]{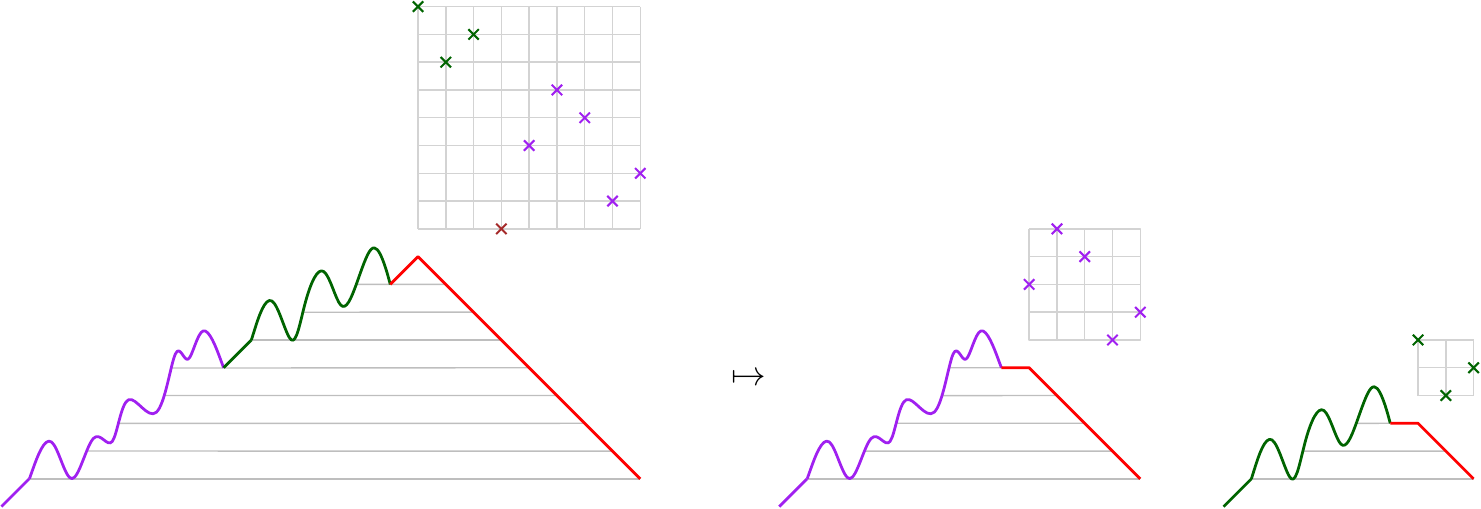}
	\caption{Recursive decomposition of an $M$-walk. We split the permutation associated with the last descending sequence into two permutations by the standard Catalan decomposition and we split the walk at the heights corresponding to the sizes of the sub-permutations.}
	\label{decompM}
\end{figure}

\lv{
\begin{note}
Using more thorough decompositions, we will show in Section~\ref{secstag} that the identity $M=zA$ holds in the more general setting of bivariate generating functions, with an extra parameter counting $0$-steps in~$M$ and staggered $\pm 2$-steps in $zA$. This will provide an alternate yet slightly longer proof of the second step above.
\end{note}

We end this section with the following observation. We just showed that $T=G$, so that~\eqref{eqBKKKKNW} also reads $1 + T + T^2 = \C$. Let us give an interpretation to this identity.

\begin{prop}
For any $n\ge 1$, the $(n-1)$-th Catalan number counts the number of $n$-vertex increasing unary-binary trees with associated permutation avoiding~$213$ that satisfy the following property: if we let 
$\ell_1$, \ldots, $\ell_k$ denote the leaves read after the last node in breadth-first search order, then the last leaf has the smallest label among the labels of $\ell_1$, \ldots, $\ell_k$.
\end{prop}

\begin{pre}
Let us consider the decorated Motzkin walk of a tree satisfying the condition of the statement. We consider its last step.
\begin{itemize}
	\item If it is a $+1$-step, then the walk is the trivial walk $(+1)$ and the corresponding tree is the vertex-tree with a leaf labeled~$1$.
	\item If it is a $0$-step, then the decorated Motzkin walk does not satisfy any further constraints (it is just a decorated Motzkin walk to which a final $0$-step is added); the corresponding tree ends with a single leaf after a unary node.
	\item If it is a $-1$-step, we let $\si=(\si_1,\ldots,\si_k)$ be the permutation assigned to the last descending sequence. The condition of the statement merely translates into $\si_k=1$. By considering the last hitting time of~$1$ before the end, the Motzkin walk can be decomposed into two successive Motzkin walks and a final $-1$-step. These Motzkin walks naturally inherit the decorations from the original Motzkin walk, with the last descending sequence of the second Motzkin walk being assigned the permutation $(\si_1-1,\ldots,\si_{k-1}-1)$.
\end{itemize}
The previous classification shows that the generating function of the objects we consider is $z+z\,T+z\,T^2=z\C$, as desired.
\end{pre}

\subsection{Unary nodes and staggered steps}\label{secstag}
In this section, we focus on the number of $n$-vertex increasing unary-binary 
trees with associated permutation avoiding $213$ that have a given number of unary nodes.
We prove Proposition~\ref{unaire} and then identify the statistic on the basketball walks that
corresponds to the number of unary nodes 
in an increasing unary-binary tree. Throughout this section, we will consider bivariate generating function of walks, with a weight~$z$ per step as above and with an extra weight~$u$ for some steps that will be specified later on. For the sake of clarity, we still denote by~$A$, $B$, $G$, $M$ and~$T$ these bivariate functions. 

\begin{proof}[Proof of Proposition~\ref{unaire}]
Recall that unary nodes of a unary-binary tree correspond to $0$-steps of the encoding Motzkin walk. By this observation and Lemma~\ref{lemMotz}, the number of $n$-vertex increasing unary-binary trees with associated permutation avoiding~$213$ having exactly $n-1-2k$ unary nodes is equal to $[u^{n-1-2k}z^n] T(z, u)$, where the variable~$u$ marks the number of $0$-steps in the Motzkin walks. 

Redoing the proof of Theorem~\ref{thmubt} with bivariate functions counting decorated Motzkin walks with an extra weight~$u$ per $0$-step, we obtain $T=M+uTM$ and $M=z\,(1+TM)^2$, hence
 $T=\psi(M)$ and $M=z\,\phi(M)$, where
\begin{equation}\label{eq:M}
\psi(y)\de \frac{y}{1-uy}\qquad\text{ and }\qquad\phi(y)\de \left(1+\frac{y^2}{1-uy}\right)^{\!2} \,.
\end{equation}
Applying the Lagrange inversion formula, we obtain
\begin{align*}
[z^n] T(z, u) 
&= \frac1n\,[y^{n-1}]\psi'(y)\phi(y)^n= \frac1n\,[y^{n-1}]\frac{1}{(1-uy)^2}\bigg(1+\frac{y^2}{1-uy}\bigg)^{\!2n}\\
&= \frac1n\,\sum_{k=0}^{\lfloor(n-1)/2\rfloor} \binom{2n}{k} [y^{n-1-2k}] \frac{1}{(1-uy)^{k+2}}\\
&= \frac1n\,\sum_{k=0}^{\lfloor(n-1)/2\rfloor} \binom{2n}{k} \binom{n-k}{k+1} u^{n-1-2k}
\end{align*}
and the result follows.
\end{proof}

We now turn to the proof of Proposition~\ref{propstag}. Recall Definition~\ref{def:staggered} and that~$\cA$ is the class of nonnegative basketball walks from~$0$ to~$0$.

\begin{pre}[Proof of Proposition~\ref{propstag}]
We denote by $A(z,u)$ the bivariate generating function of $\cA$, where $z$ marks the number of steps and $u$ the number of staggered $\pm2$-steps.
A walk in~$\cA$ is
\begin{enumerate}[label=(\textit{\roman*})]
\item either empty;\label{stagi}
\item or is the concatenation of a $+2$-step, an $A$-walk, a $-2$-step and an $A$-walk;\label{stagii}
\item or visits~$1$ (at least once) before its first return to~$0$. In that latter case,
the walk is the concatenation of a $B$-walk, an $A$-walk, a reversed $B$-walk and an $A$-walk.
We distinguish three subcases:\label{stagiii}
\begin{enumerate}[label=(\textit{\alph*})]
\item the two $B$-walks are trivial, that is, equal to $(+1)$,\label{staga}
\item exactly one of the two $B$-walks is trivial,\label{stagb}
\item both $B$-walks have length~$2$ or more.\label{stagc}
\end{enumerate}
\end{enumerate}
Now observe that, in case~\ref{stagi}, \ref{stagii}, \ref{stagiii}\ref{staga} and~\ref{stagiii}\ref{stagb}, each staggered $\pm2$-step of the original walk remains a staggered $\pm2$-step in one of its subwalks. In case~\ref{stagiii}\ref{stagc}, both $B$-walks are composed of a first $+2$-step that is not staggered in the original walk, then an $A$-walk and a reversed $B$-walk. The previous decomposition thus gives the following identity in terms of bivariate generating functions:
$$A = \underbrace{\vphantom{(z A)^2} 1}_{\ref{stagi}} + \underbrace{(z A)^2}_{\ref{stagii}} + \underbrace{(zA)^2}_{\ref{stagiii}\ref{staga}} + \underbrace{2 z A^2 (B-z)}_{\ref{stagiii}\ref{stagb}} + \underbrace{(zAB)^2 A^2}_{\ref{stagiii}\ref{stagc}}.$$
Let us express~$B$ in terms of~$A$. %Recall that a $B$-walk is a basketball walk that goes from~$0$ to~$1$ and only visits~$0$ and~$1$ at its extremities.
Following the discussion that led to~\eqref{GzA}, a $B$-walk is either the trivial walk $(+1)$, or a staggered $+2$-step followed by an $A$-walk then a reversed $B$-walk. As a result, $B=z+uzAB$, so that $B = z/(1-uzA)$. We thus obtain
\begin{align*}
zA &= z\left(1+ 2(zA)^2 + 2u \frac{(zA)^3}{1-uzA} + \frac{(zA)^4}{(1-uzA)^2}\right)\\
&= z\left(1+\frac{2(zA)^2}{1-uzA} + \frac{(zA)^4}{(1-uzA)^2}\right)
= z\left(1+\frac{(zA)^2}{1-uzA}\right)^{\!2},
\end{align*}
which is also the Lagrangian equation~\eqref{eq:M} verified by~$M$. In order to conclude the proof, it is enough to note that $G = AB=zA/(1-uzA)$, and that $T = M/(1-uM)$. 
\end{pre}

\begin{pre}[Proof of Proposition~\ref{propstagA}]
We obtained during the proof of Proposition~\ref{propstag} that $zA=z\phi(zA)$ where~$\phi$ is defined by~\eqref{eq:M}. By the Lagrange inversion formula, we have, for $n\geq 2$,
\begin{align*}
[z^{n-1}] A(z, u) 
&= [z^{n}]\, zA(z, u) =\frac1{n}\,[y^{n-1}]\phi(y)^{n}= \frac1{n}\,[y^{n-1}]\bigg(1+\frac{y^2}{1-uy}\bigg)^{\!2n}\\
&= \frac1{n}\,\sum_{k\geq 0} \binom{2n}{k} [y^{n-1-2k}] \frac{1}{(1-uy)^{k}}\\
&= \frac1{n}\,\sum_{k\geq 0} \binom{2n}{k} \binom{n-k-2}{k-1} u^{n-1-2k}
\end{align*}
and the result follows. Note that the formula $[y^a](1-y)^{-b}=\dbinom{a+b-1}{b-1}$ is not valid for $a=0$ and $b=0$; this is why we assumed $n\geq 2$.
\end{pre}

\subsection{Turning Theorem~\ref{thmubt} into an explicit bijection}

Recall that we obtained in Section~\ref{secgen} a first formal proof (via a decomposition grammar on which a few manipulations were performed) that the class~$\cC$ of basketball walks is a Catalan class. Then, in Section~\ref{secbij}, we turned this into an explicit recursive bijection with binary trees, by realizing the formal grammar manipulations as explicit walk manipulations. 

We can do the same thing in order to obtain an explicit recursive bijection between $G$-walks and increasing unary-binary trees with associated permutation avoiding $213$. In fact, we have already done in the previous sections all the necessary operations. Moreover, the study we did with bivariate generating functions in Section~\ref{secstag} will yield that the statistics counted by~$u$ in~$\cG$ will correspond to the statistics counted by~$u$ in~$\cT$.

Recall Equation~\eqref{GzA} and the proof of Theorem~\ref{thmubt}. We showed that $\cG=\cZ\cA+\cZ\cA\cG$ and that $\cT=\cM+\cM\cT$ in a combinatorially tractable way. We mean by this that we can explicitly write an object of $\cG$ as either an object of $\cZ\cA$ or a pair of an object of~$\cZ\cA$ and one of~$\cG$ and the same holds with~$\cM$ instead of~$\cZ\cA$ and~$\cT$ instead of~$\cG$. Moreover, these equations hold indifferently for univariate or bivariate generating functions. As a result, it suffices to explain how~$\cZ\cA$ and~$\cM$ can be decomposed in an isomorphic way. But this was actually done in Section~\ref{secstag}. Indeed, we combinatorially obtained
$$\cZ\cA = \cZ\left(\un+ 2(\cZ\cA)^2 + 2u (\cZ\cA)^3\sq(u\cZ\cA) + (\cZ\cA)^4\big(\sq(u\cZ\cA)\big)^2\right)\,,$$
as well as 
$$\cM = \cZ\Big(\un+\cM^2\sq(u\cM)\Big)^2\,.$$
The latter equation can be rewritten as the first one by expanding the square, which, combinatorially, means distinguishing whether each element of the pair is empty or not and by using the identity $\sq(u\cM)=\un+u\cM\sq(u\cM)$, which translates into saying that a sequence is empty or is composed of a first element and another sequence. 
}

%\phantomsection
%\addcontentsline{toc}{section}{Often used notation}
%\label{secnot}
%\printnomenclature[25mm]

\bibliographystyle{alpha}
\bibliography{main}

\newcommand{\etalchar}[1]{$^{#1}$}
\begin{thebibliography}{BKK{\etalchar{+}}16}

\bibitem[AZ07]{AyZe07}
Arvind Ayyer and Doron Zeilberger.
\newblock The number of [old-time] basketball games with final score {$n$:$n$}
  where the home team was never losing but also never ahead by more than {$w$}
  points.
\newblock {\em Electron. J. Combin.}, 14(1):Research Paper 19, 8 pp.
  (electronic), 2007.

\bibitem[Bac13]{bacher2013generalized}
Axel Bacher.
\newblock Generalized {D}yck paths of bounded height.
\newblock {\em Preprint,
  \href{http://arxiv.org/abs/1303.2724}{\nolinkurl{arXiv:1303.2724}}}, 2013.

\bibitem[BF02]{banderier2002basic}
Cyril Banderier and Philippe Flajolet.
\newblock Basic analytic combinatorics of directed lattice paths.
\newblock {\em Theoret. Comput. Sci.}, 281(1-2):37--80, 2002.

\bibitem[BKK{\etalchar{+}}16]{BKKKKNW}
Cyril Banderier, Christian Krattenthaler, Alan Krinik, Dmitry Kruchinin,
  Vladimir Kruchinin, David Nguyen, and Michael Wallner.
\newblock Explicit formulas for enumeration of lattice paths: basketball and
  the kernel method.
\newblock {\em Preprint,
  \href{http://arxiv.org/abs/1609.06473}{\nolinkurl{arXiv:1609.06473}}}, 2016.

\bibitem[BM08]{Bousquet08basket}
Mireille Bousquet-M{\'e}lou.
\newblock Discrete excursions.
\newblock {\em S\'em. Lothar. Combin.}, 57:Art. B57d, 23, 2008.

\bibitem[Duc00]{duchon2000enumeration}
Philippe Duchon.
\newblock On the enumeration and generation of generalized {D}yck words.
\newblock {\em Discrete Math.}, 225(1-3):121--135, 2000.
\newblock Formal power series and algebraic combinatorics (Toronto, ON, 1998).

\bibitem[Ges80]{gessel1980factorization}
Ira~M. Gessel.
\newblock A factorization for formal {L}aurent series and lattice path
  enumeration.
\newblock {\em J. Combin. Theory Ser. A}, 28(3):321--337, 1980.

\bibitem[Kra15]{Kra15}
Christian Krattenthaler.
\newblock Lattice path enumeration.
\newblock In {\em Handbook of enumerative combinatorics}, Discrete Math. Appl.
  (Boca Raton), pages 589--678. CRC Press, Boca Raton, FL, 2015.

\bibitem[Lot97]{lothaire1997combinatorics}
M.~Lothaire.
\newblock {\em Combinatorics on words}.
\newblock Cambridge Mathematical Library. Cambridge University Press,
  Cambridge, 1997.

\bibitem[LPRS16]{LPRS16heaps}
Derek Levin, Lara~K. Pudwell, Manda Riehl, and Andrew Sandberg.
\newblock Pattern avoidance in {$k$}-ary heaps.
\newblock {\em Australas. J. Combin.}, 64:120--139, 2016.

\bibitem[LY89]{labelle1989dyck}
Jacques Labelle and Yeong~Nan Yeh.
\newblock Dyck paths of knight moves.
\newblock {\em Discrete Appl. Math.}, 24(1-3):213--221, 1989.
\newblock First Montreal Conference on Combinatorics and Computer Science,
  1987.

\bibitem[LY90]{labelle1990generalized}
Jacques Labelle and Yeong~Nan Yeh.
\newblock Generalized {D}yck paths.
\newblock {\em Discrete Math.}, 82(1):1--6, 1990.

\bibitem[MS00]{mckenzie2000distributions}
Andy McKenzie and Mike Steel.
\newblock Distributions of cherries for two models of trees.
\newblock {\em Math. Biosci.}, 164(1):81--92, 2000.

\end{thebibliography}

\end{document}